\documentclass{article}
\usepackage[a4paper, total={6in, 8in}]{geometry}
\usepackage[utf8]{inputenc}
\usepackage{amsmath}
\usepackage{multicol}
\usepackage{graphicx}
\usepackage[colorinlistoftodos]{todonotes}
\usepackage{ stmaryrd }
\usepackage{tcolorbox}
\usepackage{algorithm, setspace}
\usepackage[noend]{algpseudocode}

\usepackage{amsfonts}
\usepackage{enumitem}
\usepackage{tikz}
\usepackage{tikz-cd}
\usetikzlibrary{decorations.markings,positioning}
\usepackage{amsthm}
\usepackage{amssymb}
\usepackage{cite}
\usepackage{float}
\usepackage{microtype}
\usepackage{mathtools}
\usepackage{hyperref}

\newcommand{\D}{\partial}
\newcommand{\R}{\mathbb{R}}

\newcommand{\Z}{\mathbb{Z}}

\newcommand{\vtl}{\vartriangleleft}

\newcommand{\Set}[1]{\{\,#1\,\}}

\newtheorem{theorem}{Theorem}[section]
\newtheorem{corollary}[theorem]{Corollary}
\newtheorem{lemma}[theorem]{Lemma}
\newtheorem{proposition}[theorem]{Proposition}
\newtheorem*{remark}{Remark}
\newtheorem*{example}{Example}
\theoremstyle{definition}
\newtheorem{definition}[theorem]{Definition}

\DeclareMathOperator{\rank}{rank}

\title{The Persistent Homology of Dual Digital Image Constructions}

\author{Bea Bleile$^{1}$, Adélie Garin$^{2}$, Teresa Heiss$^{3}$, Kelly Maggs$^{2}$, Vanessa Robins$^{4}$}

\date{%
    $^{1}$ School of Science and Technology, University of New England, Armidale, Australia\\
    $^{2}$ Laboratory for Topology and Neuroscience, École polytechnique fédérale de Lausanne (EPFL), Lausanne, Switzerland.\\
    $^{3}$ Institute of Science and Technology (IST) Austria, Kloster\-neu\-burg, Austria\\
    $^{4}$ Research School of Physics, Australian National University, Canberra, Australia.\\
}

\begin{document}
\maketitle
\begin{abstract}
To compute the persistent homology of a grayscale digital image one needs to build a simplicial or cubical complex from it. 
For cubical complexes, the two commonly used constructions (corresponding to direct and indirect digital adjacencies) can give different results for the same image.
The two constructions are almost dual to each other, and we use this relationship to extend and modify the cubical complexes to become dual filtered cell complexes. 
We derive a general relationship between the persistent homology of two dual filtered cell complexes, and also establish how various modifications to a filtered complex change the persistence diagram. 
Applying these results to images, we derive a method to transform the persistence diagram computed using one type of cubical complex into a persistence diagram for the other construction. 
This means software for computing persistent homology from images can now be easily adapted to produce results for either of the two cubical complex constructions without additional low-level code implementation. 
\end{abstract}

\section{Introduction}\label{sec:Introduction}
Persistent homology \cite{ 
Edelsbrunner2002, zomorodian2005computing} allows us to compute topological features of a space via a nested sequence of subspaces or \textbf{filtration} by returning a \textbf{persistence diagram} which reflects the connected components, tunnels, and voids appearing and disappearing as we sweep through the filtration. It has a wide range of applications in diverse contexts including digital images, used for example to study porous materials \cite{Robins_DMT_images}, hurricanes \cite{hurricanes} or in medical applications \cite{medical_applications}. 
While it is common to apply persistent homology to simplicial complexes arising from point clouds, digital images are made up of pixels (in dimension $d=2$) or voxels (for $d \geq 2$) rendering cubical complexes the natural choice as they reflect the regular grid of numbers used to encode the image.

There are two ways to construct a cubical complex from an image $\mathcal{I}$: The \textbf{V-construction} $V(\mathcal{I})$ represents voxels by vertices and the \textbf{T-construction} $T(\mathcal{I})$ represents voxels by top-dimensional cubes. These  constructions are closely related to two different voxel connectivities of classical digital topology. The V-construction corresponds to what is known in computer science as direct connectivity, where voxels are connected if and only if their grid locations differ by 1, so that each voxel has $2d$ neighbours. For $d=2$, pixels are 4-connected and the direct neighbours are to the left and right as well as above and below. The T-construction corresponds to indirect connectivity, where voxels are also connected diagonally, every voxel has $3^d-1$ neighbours and pixels are 8-connected.

It is well known that the choice of direct or indirect adjacency has an impact on the overall topological structure of a binary image and the critical points of a grayscale image function. The effects are particularly significant when the image has structure at a similar length-scale to the digital grid. It will be no surprise then that the persistent homology can also be dramatically different when computed using the V- and T-constructions for the same image, see Figure~\ref{fig:V_and_T_constructions} for an example.
A further issue in classical digital topology is that a single choice of adjacency cannot be applied to both the foreground and background of a binary image (or the sub-level and super-level sets of a grayscale image) in a topologically consistent way. If the sub-level set is given the direct adjacency, the super-level set must take the indirect adjacency (and vice versa) or the Jordan curve theorem will fail to hold, for example. 
This suggests the existence of a duality-like relationship between the V- and T-constructions applied respectively to the sub-level and super-level sets of a grayscale image.

The resulting cubical complexes are \emph{almost} dual. Vertices in the V-construction correspond to top-dimensional cells in the T-construction and interior vertices of the T-construction correspond to top-cells in the V-construction. 
This paper establishes the precise nature of the duality-like relationship between the cubical complexes of the T- and V-constructions, filtrations of these induced by an image and its negative, and their persistence diagrams. We use these results to define simple algorithms that return the persistence diagram for $V(\mathcal{I})$ 
from software that computes diagrams based on the T-construction and vice versa.  
Our results are based on a combinatorial notion of dual cell complexes and dual filtrations. 
We explain how the V- and T-constructions can be modified to obtain dual filtrations of the $d$-sphere. The relationship between the persistence diagrams of a grayscale digital image obtained via the V- and T-constructions then follows from an investigation of the effects of the required modifications on the persistence diagrams and the relationship between the persistent homology of dual filtrations on dual cell complexes. 

Our duality results make it possible to use the advantages of different software packages even when the cubical complex type is not the preferred one for the application at hand. For example, by using a streaming approach, the persistence software cubicle \cite{cubicle} is able to handle particularly large images that do not even need to fit into memory. However, it has only been implemented for the T-construction. Thanks to Section \ref{sec: T_to_V}, cubicle can now be used to compute the  V-construction persistence of an image that does not fit into memory and therefore cannot be processed by existing V-construction persistence software.

\subsection{Related Work}

At its most basic level, the algebraic relationship between the persistent homology of two dual filtered cell complexes is similar to that between persistent homology and persistent relative cohomology. 
The latter corresponds to taking the anti-transpose of the boundary matrix~ \cite{Vin} or equivalently, leaving the boundary matrix as is and applying the row reduction algorithm instead of the column-reduction algorithm~\cite{Vin,oelsboeck}. 
Lemma~\ref{minor_diagonal_lemma} shows that the same relationship applies to the boundary matrices of dual filtered cell complexes. Therefore, Theorem~\ref{main_theorem} can be viewed as a translation of the known bijection between the persistence pairs of persistent homology and relative persistent cohomology into the setting of dual filtrations. 
This theorem is the first step towards establishing the mapping between persistence diagrams of the T- and the V-construction of images in Section~\ref{image_results}. Furthermore, it applies more generally to the persistent homology of dual filtered cell complexes without using the connection to persistent relative cohomology.

The symmetry of extended persistence diagrams~\cite{extending_pers} is also closely related to our Theorem~\ref{main_theorem}.
The mathematical setting for extended persistence is a filtered simplicial complex whose underlying space, $X$, is a manifold without boundary.  It extends the homology sequence derived from a filtration of $X$ by the sub-level sets of a function $X_s = (f^{-1}(\infty,s])$, to continue with the relative homology of the pair $(X,f^{-1}[r,\infty))$, where $s$ is an increasing threshold and $r$ is a decreasing one. 
As observed in~\cite{adaptative}, the cubical complex constructions used in digital image analysis do not provide a suitable topological structure for extended persistence because of the inconsistencies that arise if both the sub- and super-level sets are treated as closed.
The authors of~\cite{adaptative} overcome this by constructing a simplicial complex from the digital image that consistently reflects the connectivity of both sub- and super-level sets, and use this to obtain the expected symmetries in the extended persistence diagram. 
In contrast, we work with the existing widely-implemented cubical complex constructions of digital images and establish results that permit a simple high-level algorithm to transform 
between two regular (not extended) persistence diagrams. 

This paper is a follow-up of an extended abstract \cite{socg} published in the Young Researcher Forum of SoCG $2020$. A version with appendices included is available on arXiv. 
These results have already been cited as the basis for an algorithm implemented by the developers of the software cubical ripser\cite{cubical_ripser}.


\subsection{Overview}

Our results are aimed at both pure and applied mathematicians who want to understand and use the relationship between the persistent homology of dual filtered cell complexes and particularly the two standard constructions of cubical complexes from digital images.

In Section~\ref{sec:background} we define dual cell complexes and dual filtered complexes as used throughout this paper, and provide a brief outline of the definitions and results of persistent homology. The reader who is not familiar with the theory of persistent homology should refer to~\cite{Edel08,zomorodian2005computing} for a more comprehensive introduction. Section~\ref{sec:main_duality_result} establishes the relationship between persistence diagrams of two dual filtered cell complexes. 

In Section \ref{sec:grayscale_digital_image_constructions}, we describe and formalise the two standard cubical complexes used in topological computations on digital images. 
We explain how these two complexes (the T- and V-constructions described earlier) must be extended and modified  to form dual filtered cell 
complexes with underlying space homeomorphic 
to the $d$-sphere. 
The effects these modifications have on persistence diagrams are derived in Section~\ref{sec:technical}. 
For the investigation of one of these effects we use the long exact sequence of a filtered pair of cell complexes arising from the category theoretic view of persistence modules.

The last Section \ref{image_results} states the results for persistence diagrams of digital images and explains how to compute the persistence diagram of the T-construction by simple manipulation of a persistence diagram computed using the V-construction, and vice versa. This gives a practical method for adapting the output from existing software packages that use one or the other construction to obtain the persistence diagram for the dual construction.

\section{Mathematical Background}\label{sec:background}

\subsection{Dual Cell Complexes and Filtrations}

CW-complexes \cite{lundell2012topology} generalize simplicial complexes to allow cells that are not necessarily simplices but homeomorphic to open discs or balls, for example cubes instead of tetrahedra. A CW-complex is \textbf{regular} if the closure of each $k$-cell is homeomorphic to the closed $k$-dimensional ball $D^k$. For the remainder of this paper a \textbf{cell complex} is a finite regular  CW-complex. The \textbf{dimension} $\dim(X)$ of the cell complex $X$ is the maximum dimension of cells in $X$ and,  
writing $X$ for the topological space as well as for the set of its cells, we obtain $\dim(X) = \max\{\dim(\sigma)\, | \, \sigma \in X\}$.

Let $X$ be a cell complex with cells $\tau$ and $\sigma$. If $\tau \subseteq \overline{\sigma}$ then $\tau$ is a \textbf{face} of $\sigma$, and $\sigma$ is a \textbf{coface} of $\tau$, written as $\tau \preceq \sigma$. The \textbf{codimension} of a pair of cells $\tau \preceq \sigma$ is the difference in dimension, $\dim(\sigma) - \dim(\tau)$. If $\sigma$ has a face $\tau$ of codimension 1, we call $\tau$ a \textbf{facet} of $\sigma$, and write $\tau \vtl \sigma$. A function $f : X \to \R$ on the cells of $X$ is \textbf{monotonic} if $f(\sigma) \leq f(\tau)$ whenever $\sigma \preceq \tau$.

\begin{definition}\label{combinatorial_dual}
    The $d$-dimensional cell complexes $X$ and $X^*$ are \textbf{combinatorially dual} if there is a bijection $X \rightarrow X^*, \sigma \mapsto \sigma^*$ between the sets of cells such that
    \begin{enumerate}
    \item (Dimension Reversal) $\dim(\sigma^*) = d - \dim\sigma$ for all $\sigma \in X$.
    \item (Face Reversal) $\sigma \preceq \tau \iff \tau^* \preceq \sigma^*$ for all $\sigma, \tau \in X$.
    \end{enumerate}
\end{definition}

\begin{definition} \label{def:filtered_complex}
A \textbf{filtered (cell) complex} $(X,f)$ is a cell complex $X$ together with a monotonic function $f:X \to \R$. A linear ordering $\sigma_0, \sigma_1, \dots, \sigma_n$ of the cells in $X$, such that $\sigma_i \preceq \sigma_j$ implies $i \leq j$, is \textbf{compatible} with the function $f$ when $$ f(\sigma_0) \leq f(\sigma_1) \leq \ldots \leq f(\sigma_n).$$
\end{definition}

Note that the monotonicity condition implies that, for $r \in \mathbb R$, the sub-level set  
$$ X_r :=  f^{-1}(-\infty,r]$$
is a subcomplex of $X$. The value $f(\sigma)$ determines when a cell enters the filtration given by this nested sequence of subcomplexes. The definition of a  compatible ordering also implies that each step in the sequence
$$\emptyset \subset \Set{\sigma_0} \subset \Set{\sigma_0, \sigma_1} \subset \dots \subset \Set{\sigma_0, \sigma_1, \dots, \sigma_n}=X$$
is a subcomplex, and every sub-level set $f^{-1}(-\infty,r]$ appears somewhere in this sequence: $ f^{-1}(-\infty,r] = f^{-1}(-\infty,f(\sigma_i)] = \Set{\sigma_0, \sigma_1, \dots, \sigma_i}$ for $i=\max \Set{ i=0,\dots,n \ | \ f(\sigma_i)\leq r}$.

\begin{definition}
    Two filtered complexes $(X,f)$ and $(X^*,g)$ are \textbf{dual filtered complexes} if $X$ and $X^*$ are combinatorially dual to one another and if there exists a linear ordering $\sigma_0, \sigma_1, \ldots, \sigma_n$ of the cells in $X$ that is compatible with $f$ and its dual ordering $\sigma^*_n, \sigma^*_{n-1}, \ldots, \sigma^*_0$ is compatible with $g$.
\end{definition}

\begin{proposition} \label{negative_function_lemma}
    Suppose two functions $f : X \to \R$ and $f^* : X^* \to \R$ satisfy $f^*(\sigma^*) = -f(\sigma)$. Then $(X,f)$ and $(X^*,f^*)$ are dual filtered complexes.
\end{proposition}

\subsection{Persistent Homology}

When working with data, standard topological quantities can be highly sensitive to noise and small geometric fluctuations. Persistent homology addresses this problem by examining a collection of spaces, indexed by a real variable often representing an increasing length scale.
These spaces are modelled by a cell complex $X$ with a filter function $f:X\to \R$ assigning to each cell the scale at which this cell appears.

\subsubsection{Definition}
Given a filtered complex $(X,f)$, we obtain inclusions $f^{-1}(-\infty,r] \to f^{-1}(-\infty,s]$ of sub-level sets for $r \leq s$.
Applying degree-$k$ homology with coefficients in $\Z/2\Z$ to these inclusions yields linear maps between vector spaces
$$H_k(f^{-1}(-\infty,r]) \to H_k(f^{-1}(-\infty,s]).$$
The resulting functor $H_k(f):(\R, \leq) \to \mathsf{Vec}_{\Z/2\Z}$ from the poset category $(\R, \leq)$ to the category of vector spaces over the field $\Z/2\Z$ is called a \textbf{persistence module}, for details see \cite{chazal2016structure}.

As discussed in \cite{chazal2016structure}, Gabriel's Theorem from representation theory implies that the persistence module $H_k(f)$ decomposes into a sum of persistence modules consisting of $\mathbb{Z}/2\mathbb{Z}$ for $r\in[b,d)$ connected by identity maps, and $0$ elsewhere, called \textbf{interval modules} $\mathbb{I}_{[b,d)}$:
$$H_k(f) \cong \bigoplus_{l\in L} \mathbb{I}_{[b_l, d_l)}.$$
Each interval summand $\mathbb{I}_{[b_l,d_l)}$ represents a degree-$k$ homological feature 
that is \textbf{born} at $r=b_l$ and \textbf{dies} at $r=d_l$. If the final space $X$ has non-trivial homology there are features that never die. These have $d_l=\infty$ and the interval is called \textbf{essential}.

The \textbf{degree-$k$ persistence diagram} of $f$ is the multiset $$\mathsf{Dgm}^k(f) = \Set{ [b_l,d_l) \mid 
l \in L}.$$
We write $[b_l, d_l)_k \in \mathsf{Dgm}^k(f)$ to denote the homological degree of an interval and define the \textbf{persistence diagram} of $f$ as the disjoint union over all degrees: $$\mathsf{Dgm}(f) = \bigsqcup_{k=0}^{\textrm{dim}(X)} \mathsf{Dgm}^k(f).$$
Writing $\mathsf{Dgm}_\mathbf{F}(f)$ for the multiset of finite intervals with $d_l < \infty$, and $\mathsf{Dgm}_\infty(f)$ for the remaining essential ones, we obtain $\mathsf{Dgm}(f) = \mathsf{Dgm}_\mathbf{F}(f) \sqcup \mathsf{Dgm}_\infty(f)$. 

\subsubsection{Computation}
\label{sec:computations}
To compute the persistence diagram $\mathsf{Dgm}(f)$ we choose an ordering $\sigma_0, \sigma_1, \dots, \sigma_n$ of the cells in $X$ that is compatible with $f$. Cells $\sigma_i$ and $\sigma_j$ appear at the same step in the nested sequence of sub-level sets $\left(f^{-1}(\infty,r]\right)_{r \in \R}$ if $f(\sigma_i)=f(\sigma_j)$. 
For the following computations however, we must add exactly one cell at every step:
$$\emptyset \subset \Set{\sigma_0} \subset \Set{\sigma_0, \sigma_1} \subset \dots \subset \Set{\sigma_0, \sigma_1, \dots, \sigma_{n-1}} \subset \Set{\sigma_0, \sigma_1, \dots, \sigma_n}=X.$$
When adding the cells one step at a time, a cell of dimension $k$ causes either the birth of a $k$-dimensional feature or the death of a $(k-1)$-homology class~\cite{delfinado1995incremental}, that is, each cell is either a \textbf{birth} or a \textbf{death cell}. 
A pair $(\sigma_i, \sigma_j)$ of cells where $\sigma_j$ kills the homological feature created by $\sigma_i$ is called a \textbf{persistence pair}. 
A persistence pair $(\sigma_i, \sigma_j)$ corresponds to the interval $[f(\sigma_i),f(\sigma_j)) \in \mathsf{Dgm}_\mathbf{F}(f)$. Note that this interval can be empty, namely if $f(\sigma_i)=f(\sigma_j)$. Empty intervals are usually neglected in the persistence diagram. A birth cell $\sigma_i$ with no corresponding death cell is called \textbf{essential}, and corresponds to the interval $[f(\sigma_i),\infty) \in \mathsf{Dgm}_\infty(f)$.

Recall that presentations for the standard homology groups are found by studying the image and kernel of integer-entry matrices that represent the boundary maps taking oriented chains of dimension $k$ to those of dimension $(k-1)$~\cite{munkres}. 
In persistent homology, we work with the $\Z/2\Z$ \textbf{total boundary matrix} $D$, which is defined by $D_{i,j} = 1$ if $\sigma_i \vtl \sigma_j$ and $0$ otherwise.
Define $$r_D(i,j) = \rank D_i^j - \rank D_i^{j-1} - \rank D_{i+1}^j + \rank D_{i+1}^{j-1}$$
where $D_i^j =D[i:n, 0:j]$ is the lower-left sub-matrix of $D$ attained by deleting the first rows up to $i-1$ and the last columns starting from $j+1$.

\begin{theorem}[Pairing Uniqueness Lemma \cite{vineyards}] \label{uniqueness_of_pairing}
    Given a linear ordering of the cells in a filtered  cell complex $X$, $(\sigma_i, \sigma_j)$ is a persistence pair if and only if $r_D(i,j) = 1$.
\end{theorem}

The ranks are usually computed by applying the column reduction algorithm~\cite{Edel08} to obtain the %
reduced matrix $R$ and using the property that $\rank D_i^j = \rank R_i^j$ under the operations of the algorithm. The persistence pairs can then be read off easily since $r_R(i,j) = 1$ if and only if the $i$th entry of the $j$th column of the reduced matrix is the lowest 1 of this column. However, in this paper, we can work directly with $r_D$.

\begin{corollary} \label{essential_corollary}
    If $r_D(i,j) \neq 1$ and $r_D(j,i) \neq 1$ for all $j$ then the cell $\sigma_i$ is essential.
\end{corollary}
\begin{proof}
The fact that every cell is either a birth or a death cell implies that $\sigma_i$ must be an unpaired birth or death cell. However, as every filtration begins as the empty set, there are no unpaired death cells.
\end{proof}

\section{ The Persistent Homology of Dual Filtered Complexes}
\label{sec:main_duality_result} 
 
Recall again that in standard homology and cohomology the coboundary map is the adjoint of the boundary map. Hence, given a consistent choice of bases for the chain and cochain groups, their matrix representations are related simply by taking the transpose.  
In~\cite{Vin}, another algebraic relationship is established between persistent homology and persistent relative cohomology, based on the observation that the filtration for relative cohomology reverses the ordering of cells in the total (co)boundary matrix.
The same reversal of ordering holds for the dual filtered cell complexes defined here, so we obtain a similar relationship between the persistence diagrams. 
Our proof of the correspondence between persistence pairs in dual filtrations uses the matrix rank function and pairing uniqueness lemma in a similar way to the combinatorial Helmoltz-Hodge decomposition of~\cite{oelsboeck}. 
Nonetheless, Theorem~\ref{main_theorem} interprets the underlying linear algebra in the setting of dual filtered complexes and only uses the concept of persistent homology without using the connection to persistent relative cohomology, which makes it more accessible. 
 
For this section suppose $(X,f)$ and $(X^*,g)$ are dual filtered cell complexes with $n+1$ cells. 
Suppose that a linear ordering $\sigma_0, \sigma_1, \ldots, \sigma_n$ 
of the cells in $X$ is compatible with the filtration $(X,f)$, and that $\sigma^*_n, \sigma^*_{n-1}, \ldots, \sigma^*_0$ is the dual linear ordering compatible with $g$. Let $D$ be the total boundary matrix of $X$ 
and $D^*$ be the total boundary matrix of $X^*$ with their respective orderings. 

\begin{remark}
    A useful indexing observation is that $\sigma^*_i$ is the $(n-i)$-th cell of the dual filtration.
\end{remark}

We denote by $D^\perp$ the anti-transpose of the matrix $D$, that is the reflection across the minor diagonal: $D^\perp_{i,j} = D_{n-j,n-i}.$  Anti-transposition is also the composition of standard matrix transposition with a reversal of the order of the columns and of the rows.

\begin{lemma} \label{minor_diagonal_lemma}
    The matrix $D^*$ is the anti-transpose $D^\perp$ of $D$, that is, $$D^*_{i,j} = D_{n-j,n-i} = D^{\perp}_{i,j}.$$
\end{lemma}

\begin{proof}
    The equivalences below follow from the definition of $D$, of dual cell complexes, and the above remark.    
    $$D_{n-j,n-i} = 1 \Leftrightarrow \sigma_{n-j} \vtl  
    \sigma_{n-i} \Leftrightarrow \sigma^*_{n-i} \vtl  \sigma^*_{n-j}  \Leftrightarrow D^*_{i,j} = 1.$$
\end{proof}

\begin{lemma} \label{rank_lemma}
The sub-matrices defined in Section~\ref{sec:computations} satisfy  $$ (D_i^j)^\perp = (D^\perp)_{n-j}^{n-i} $$ and thus
   $$ \rank \, D_i^j = \rank\, (D^\perp)_{n-j}^{n-i}$$ and  
   $$ r_D(i,j) = r_{D^\perp}(n-j,n-i). $$
\end{lemma}

\begin{proof}
The first statement follows from $$(D_i^j)^\perp = (D[i:n, 0:j])^\perp = D^\perp[(n-j):n, 0:(n-i)] = (D^\perp)_{n-j}^{n-i}. $$ 
The second statement follows because anti-transposition is attained by composing the rank preserving operations of  transposition and row and column permutations. The third statement follows from the second through: 
\begin{align*}
 r_D(i,j) &= \rank D_i^j - \rank D_i^{j-1} - \rank D_{i+1}^j + \rank D_{i+1}^{j-1} \\ 
 &= \rank (D^\perp)_{n-j}^{n-i} - \rank (D^\perp)_{n-j+1}^{n-i} - \rank (D^\perp)_{n-j}^{n-i-1} + \rank (D^\perp)_{n-j+1}^{n-i-1} \\
 &= r_{D^\perp}(n-j,n-i).   
\end{align*}
\end{proof}

\begin{theorem}[Persistence of Dual Filtrations] \label{main_theorem}
    Let $(X, f)$ and $(X^*,g)$ be dual filtered complexes with compatible ordering $\sigma_0, \sigma_1, \dots, \sigma_n$. Then
    \begin{enumerate}
        \item $(\sigma_i, \sigma_j)$ is a persistence pair in the filtered complex $(X,f)$ if and only if $(\sigma^*_j, \sigma^*_i)$ is a persistence pair in $(X^*,g)$.
        \item $\sigma_i$ is essential in $(X,f)$ if and only if $\sigma_i^*$ is essential in $(X^*,g)$.
    \end{enumerate}
\end{theorem}

\begin{proof}
    Lemma \ref{rank_lemma} implies that $r_{D}(i,j) = r_{D^*}(n-j, n-i)$. Therefore, $$r_{D}(i,j) = 1 \Leftrightarrow r_{D^*}(n-j, n-i) = 1.$$
    By the Pairing Uniqueness Lemma \ref{uniqueness_of_pairing}, the above implies that $(\sigma_i, \sigma_j)$ is a persistence pair whenever the $(n-j)$-th cell of the dual filtration $(X^*, g)$ is paired with the $(n-i)$-th, thus proving Part (1). For Part (2), Lemma \ref{rank_lemma} also tells us that the following two statements are equivalent:
    \begin{itemize} 
        \item Both $r_D(i,j) \neq 1$ and $r_D(j,i) \neq 1$ for all $j$.
        \item  Both $r_{D^*}(n-j, n-i) \neq 1$ and $r_{D^*}(n-i, n-j) \neq 1$ for all $n-j$.
    \end{itemize}
    By Corollary \ref{essential_corollary}, this means that $\sigma_i$ is an essential cell in $(X,f)$ if and only if the $(n-i)$-th cell $\sigma^*_i$ is essential in the dual filtration $(X^*,g)$.
\end{proof}

\begin{corollary}
    Let $(X,f)$ and $(X^*,g)$ be dual filtered complexes. Then 
    \begin{enumerate}
        \item $$[f(\sigma_i), f(\sigma_j)) \in \mathsf{Dgm}^k_\mathbf{F}( f) \Leftrightarrow [g(\sigma^*_j), g(\sigma^*_i)) \in \mathsf{Dgm}^{d-k-1}_\mathbf{F}(g).$$
        \item $$[f(\sigma_i), \infty) \in \mathsf{Dgm}^k_\infty(f) \Leftrightarrow [g(\sigma^*_i), \infty) \in \mathsf{Dgm}^{d-k}_\infty(g).$$
    \end{enumerate}
\end{corollary}

\begin{proof}
    Note that for a persistence pair $(\sigma_i, \sigma_j)$,
    found for an ordering compatible with the function $f$, 
    the birth value is $f(\sigma_i)$ and the death value is $f(\sigma_j)$. The result then follows directly from Theorem \ref{main_theorem}.
\end{proof}

\begin{remark}
It is worth  noting that there is a dimension shift between essential and non-essential pairs coming from the fact that the birth cell defines the dimension of a homological feature. 
For finite persistence pairs, the birth cell changes from $\sigma_i$ (of dimension $k$) to $\sigma_j^*$ (of dimension $d-(k+1)$) in the dual, while for an essential cycle, the birth cell in the dual is $\sigma^*_i$. 
This dimension shift also appears in our results on images later on. 
\end{remark}

\section{Filtered Cell Complexes from Digital Images} 
\label{sec:grayscale_digital_image_constructions}

As described in the introduction, the motivating application for the duality results of this paper is grayscale digital image analysis. This section begins with the definition of grayscale digital images and describes the two standard ways to model such images by  cubical complexes as well as the modifications required to make these dual filtered complexes. 

\begin{definition}
A \textbf{$d$-dimensional grayscale digital image} of size $(n_1, n_2, \ldots, n_d)$ is an $\R$-valued array $\mathcal{I} \in M_{n_1 \times n_2 \times \ldots \times n_d}(\R)$. Equivalently, it is a real-valued function on a $d$-dimensional rectangular grid $$\mathcal{I} : I = \llbracket 1, n_1\rrbracket \times \llbracket1, n_2\rrbracket \times \ldots \times \llbracket1, n_d\rrbracket \to \R$$ where $\llbracket1, n_i\rrbracket$ is the set $\Set{ k \in \mathbb{N} \mid 1 \leq k \leq n_i}$. The index set, $I$, of $\mathcal{I}$ is also called the \textbf{image domain}.  
\end{definition} 

Recall that elements $p \in I$ are called \textbf{pixels} when $d$=2, \textbf{voxels} if $d\geq 3$,  and the value $\mathcal{I}(p) \in \R$ is the \textbf{grayscale value} of $p$.

One would like to use persistent homology to analyse such images via their sub-level sets. However, the canonical topology on $I \subseteq \Z^d \subset \R^d$ makes it a totally disconnected discrete space. To induce a meaningful topology on the image that better represents the perceived connectivity of the voxels, grayscale digital images are modelled by regular cubical complexes \cite{kovalevsky_finite_1989}. 

\begin{definition}
\label{cubcomp}
An \textbf{elementary $k$-cube} $\sigma \subset \R^{d}$ is the product of $d$ elementary intervals, $$\sigma = e_1 \times e_2  \times \ldots \times e_d$$ such that $k$ of the intervals have the form  $e_i = [l_i,l_i+1]$ and $d-k$ are degenerate, $e_i = [l_i,l_i]$. 

A \textbf{cubical complex} $X \subset \R^d$ is a cell complex consisting of a set of elementary $k$-cubes, such that all faces of $\sigma \in X$ are also in $X$, and such that all vertices of $X$ are related by integer offsets.
\end{definition}

\subsection{Top-cell and Vertex Constructions}

There are two common ways to build a filtered cubical complex from an image $\mathcal{I}: I \longrightarrow \R$. One method is to represent the voxels as vertices of the cubical complex
as in \cite{Robins_DMT_images}. We call this cubical complex the vertex construction, or V-construction for short. 
The second method takes voxels as top-dimensional cells: we call it the  top-cell construction, or T-construction. 
It is shown in \cite{kovalevsky_finite_1989}, that the vertex construction corresponds to the graph-theoretical direct adjacency used in traditional digital image processing and the top-cell construction to the indirect adjacency model. These adjacency models are also respectively referred to as the open and closed digital topologies.
 
An example of how each construction is built from an image is given in Figure \ref{fig:V_and_T_constructions}. 
The explicit definitions of such constructions are given below. 

\begin{figure}
    \centering
    \includegraphics[width=\linewidth]{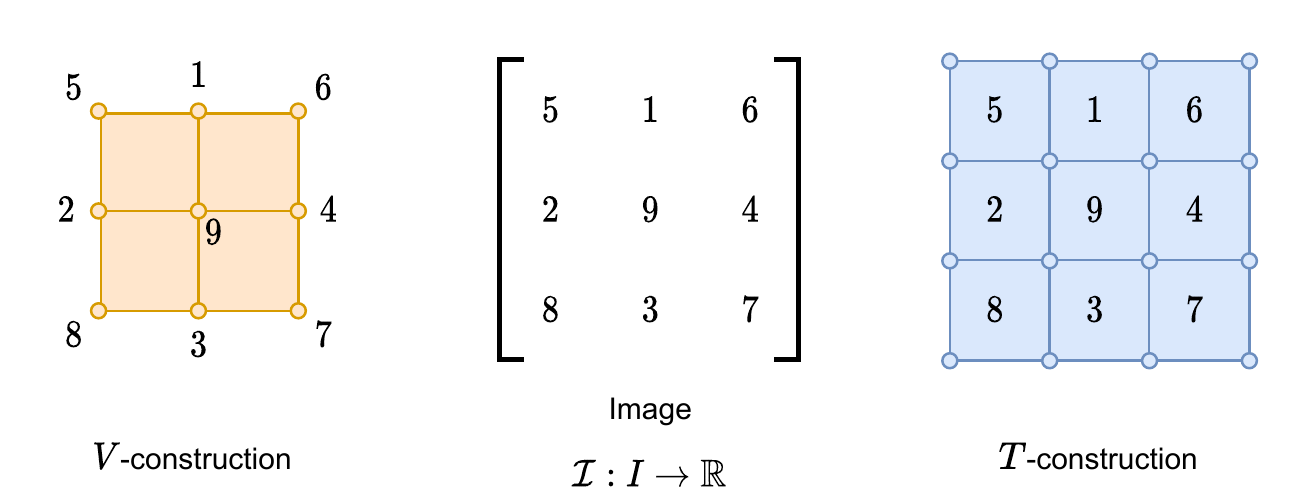} \vspace{3em} \\
    \includegraphics[scale=0.25]{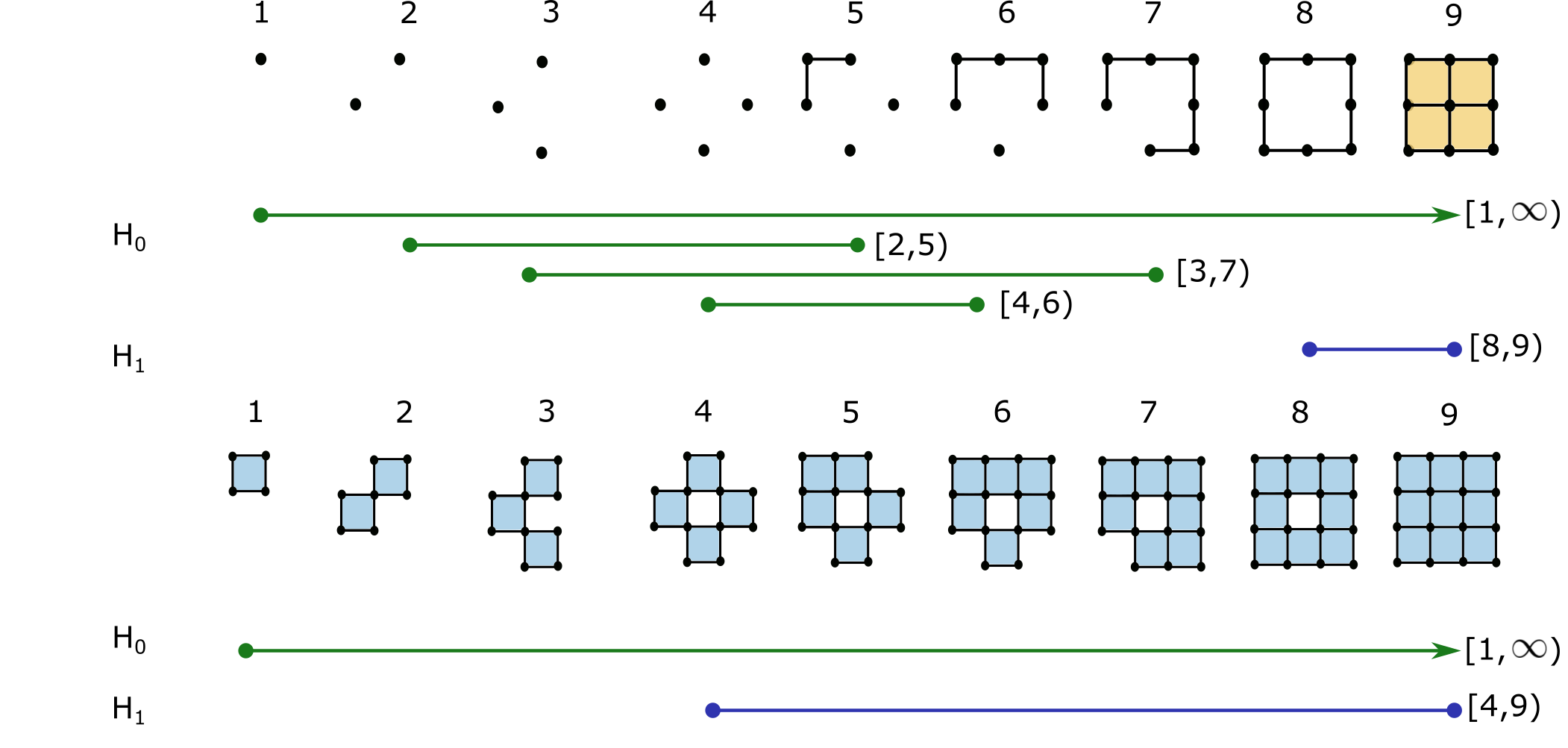}
    \caption{Top: The V- and T-constructions generated by an image $\mathcal{I} : I \to \R$ with the values of $\mathcal{I}$ indicated on the vertices and the top-dimensional cells, respectively.\newline  Middle: the 
    filtration 
    $V(\mathcal{I}) : V(I) \to \R$ and the corresponding persistence pairs.\newline Bottom: the 
    filtration 
    $T(\mathcal{I}) : T(I) \to \R$ and the corresponding persistence pairs.}
    \label{fig:V_and_T_constructions}
\end{figure}

\begin{definition}
     Given a $d$-dimensional grayscale digital image $\mathcal{I}: I \to \R$, of size $(n_1,n_2,\dots,n_d)$, 
     the \textbf{V-construction} is a filtered cell complex $(V(I),V(\mathcal{I}))$ defined as follows.
     \begin{enumerate}
         \item $V(I)$ is a cubical complex built from an array of $(n_1 - 1) \times \ldots \times (n_d-1)$ elementary $d$-cubes and all their faces. 
         \item The vertices $\upsilon^{(0)} \in V(I)$ are indexed exactly by the elements $p \in I$, and we define the function $V(\mathcal{I})$ firstly on these vertices as, $$V(\mathcal{I})(\upsilon^{(0)}) = \mathcal{I}(p).$$
         Then for an elementary $k$-cube $\sigma$, the function takes the maximal value of its vertices 
         $$ V(\mathcal{I})(\sigma) = \max\limits_{\upsilon^{(0)} \preceq \sigma} V(\mathcal{I})(\upsilon^{(0)}).$$
         This ensures that $V(\mathcal{I})$ is monotonic with respect to the face relation on $V(I)$. 
     \end{enumerate}
\end{definition} 

\bigskip

\begin{definition}
     Given a $d$-dimensional grayscale digital image $\mathcal{I}: I \to \R$, of size $(n_1,n_2,\dots,n_d)$, 
     the \textbf{T-construction}  is a filtered cell complex $(T(I),T(\mathcal{I}))$ defined as follows.
     \begin{enumerate}
         \item $T(I)$ is a cubical complex built from the array of $n_1 \times \ldots \times n_d$ elementary $d$-cubes and all their faces. 
         \item The $d$-cells $\tau^{(d)} \in T(I)$ are indexed exactly by the elements $p \in I$, and we define the function $T(\mathcal{I})$ firstly on these top-dimensional cells as, $$T(\mathcal{I})(\tau^{(d)}) = \mathcal{I}(p).$$
         Then for an elementary $k$-cube $\sigma$, the function takes the smallest value of any adjacent $d$-cubes, 
         $$ T(\mathcal{I})(\sigma) = \min\limits_{\sigma \preceq\tau^{(d)}} T(\mathcal{I})(\tau^{(d)}).$$ 
         This ensures that $T(\mathcal{I})$ is monotonic with respect to the face relation on $T(I)$. 
     \end{enumerate}
\end{definition} 

The next section describes how to modify the original image and take quotients to obtain dual complexes and filtrations. 

\subsection{Modifications for Duality} \label{sec:duality_complexes}

The cubical complexes defined using the top-cell and vertex constructions are not strictly dual to each other in the standard context of a rectangular digital image domain due to the presence of a boundary. There are two methods to resolve this issue. One is to treat the image domain as periodic and identify opposite faces of the rectangular domain. 
This makes $V(I)$ and $T(I)$ into dual cubical complexes with $d$-cubes and vertices in both cases indexed by $I$; their underlying space is the $d$-torus. 
Taking $V(\mathcal{I})$ as the function on $V(I)$ and 
$T(-\mathcal{I})$ as the function on $T(I)$, we also obtain dual filtrations and Theorem~\ref{main_theorem} can be applied to deduce the persistence pairs of one filtered complex 
from the other.

A more commonly used approach to handling the boundary of a convex domain in $\mathbb R^d$ is to take the quotient identifying the boundary to a point and thus treat the convex domain as a subset of the $d$-sphere. 
To obtain dual cell complexes on the $d$-sphere, we 
increase the size of the image domain before taking the quotient modulo the boundary. 
The image function is assigned a large arbitrary value on these extra voxels and dual filtered complexes are obtained by considering $\mathcal{I}$ in one construction and $-\mathcal{I}$ in the other as detailed in the following definitions and results.

Let $\mathcal{I} : I \to \R$ be a grayscale digital image with index set $I = \llbracket 1, n_1\rrbracket \times \ldots \times \llbracket 1, n_d\rrbracket $, and set $$ N > \max_{p \in I} \, \mathcal{I}(p).$$ 
\begin{definition}
    The \textbf{padded image} $\mathcal{I}^{\mathbf{P}} : I^{\mathbf{P}} \to \R$ has image domain $I^{\mathbf{P}} = \llbracket 0, n_1+1\rrbracket \times \ldots \times \llbracket 0, n_d +1 \rrbracket$ and 
    $$\mathcal{I}^{\mathbf{P}}(p) = \begin{cases} \mathcal{I}(p), & \text{for $p \in I$} \\
    N, & \text{for $p \in I^{\mathbf{P}}\setminus I$} 
    \end{cases}
    $$
\end{definition}

As shown in Figure~\ref{fig:overlay}, the padded image is simply obtained by adding a shell of $N$-valued voxels to $\mathcal{I}$. 
We denote the V- and T-constructions over the padded image by $V(\mathcal{I}^{\mathbf{P}}) : V(I^\mathbf{P}) \to \R$ and $T(\mathcal{I}^{\mathbf{P}}) : T(I^\mathbf{P})  \to \R$, respectively.

\begin{figure}[H]
 \centering
 \includegraphics[width=\linewidth]{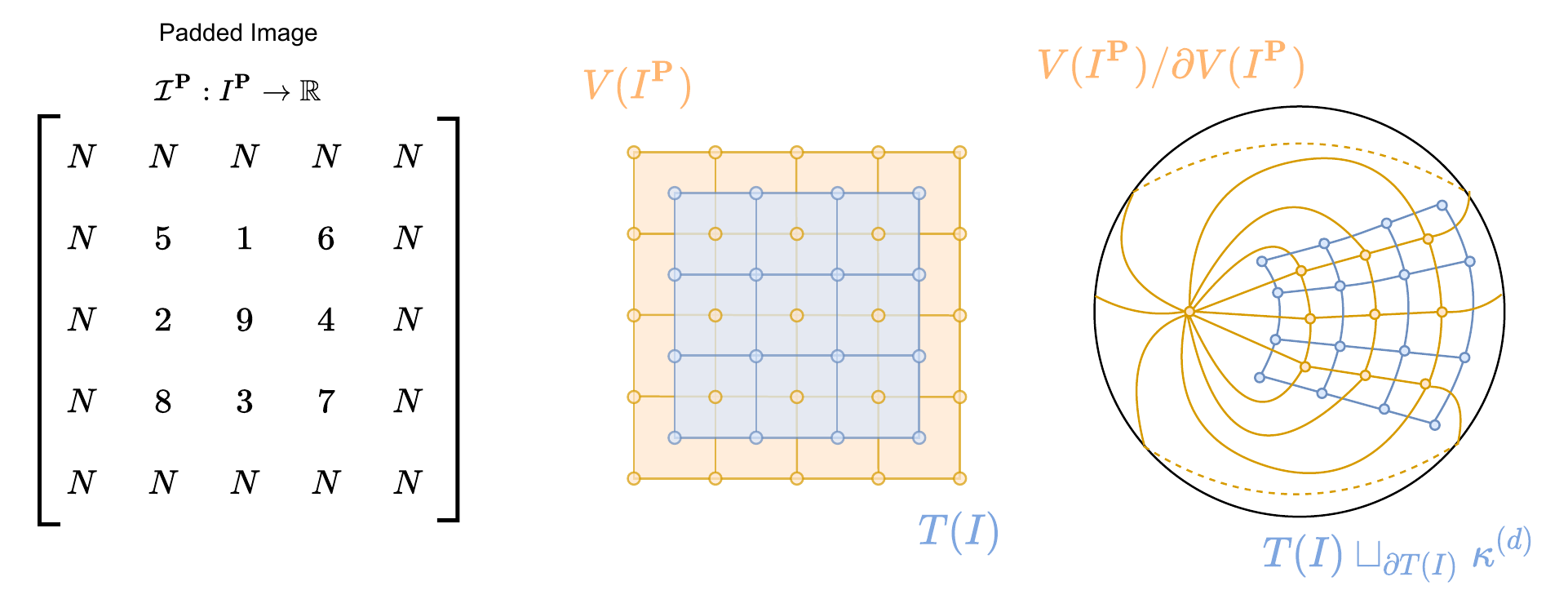}
 \caption{The transformation of the V- and T-construction into dual cell complexes (right) using the padded image (left) and a mapping from $V(I^\mathbf{P})$ to $T(I)$ (center).}
 \label{fig:overlay}
\end{figure}
 
Let $T(I) \sqcup_{\partial T(I)} \kappa^{(d)}$ denote the cell complex obtained from  $T(I)$ by attaching a $d$-cell $\kappa^{(d)}$ along the boundary $\partial T(I)$ (see \cite[p.5]{hatcher2002} for details). Furthermore, let $V(I^\mathbf{P}) / \partial V(I^\mathbf{P})$ be the quotient cell complex obtained by identifying all points of the boundary (see \cite[p.8]{hatcher2002} for details).
Note that both these modifications create cells that are not elementary cubes.
 
\begin{lemma}
    Given a rectangular digital image domain, $I$, the quotient, $V(I^\mathbf{P}) / \partial V(I^\mathbf{P})$, of the padded V-construction modulo its boundary is the combinatorial dual of $T(I) \sqcup_{\partial T(I)} \kappa^{(d)}$. 
\end{lemma}

\begin{proof}
Each elementary $k$-cube, $\sigma \in V(I^{\mathbf{P}})$ takes the form $$ \sigma = e_1 \times \ldots \times e_d, \quad e_i = [l_i, l_i + 1] \text{ or } e_i = [p_i, p_i] $$ 
where $k$ of the elementary intervals are non-degenerate with $l_i \in \{0, \ldots, n_i \}$ and $(d-k)$ are degenerate with $p_i \in \{0, \ldots, n_i +1 \}$. 
Note that $\sigma \in \partial V(I^{\mathbf{P}})$ if at least one degenerate interval has $p_i = 0$ or $(n_i+1)$. 
Now consider the following cell constructed from  $\sigma$: $$\sigma^* = e_1^* \times \ldots \times e_d^*, \quad e_i^* = [l_i + \tfrac{1}{2} , l_i +\tfrac{1}{2} ] \text{ or } [p_i-\tfrac{1}{2}, p_i + \tfrac{1}{2}] $$ 
with $l_i$ and $p_i$ as defined above.
This cell has $k$ degenerate intervals and $(d-k)$ non-degenerate ones so $\sigma^*$ is an elementary $(d-k)$-cube. 
If we insist that $\sigma \not\in \partial V(I^{\mathbf{P}})$, then we see that $p_i \in \{1, \ldots, n_i \}$, and the degenerate coordinate values $(l_i+\tfrac{1}{2}) \in \{ \tfrac{1}{2}, \tfrac{3}{2} \ldots, (n_i + \tfrac{1}{2})\}$. 
Thus we obtain a bijection between $k$-cells in $V(I^{\mathbf{P}}) \setminus \partial V(I^{\mathbf{P}})$ and $(d-k)$-cells in $T(I)$. Mapping the $0$-cell $[\partial V(I^\mathbf{P})] \in V(I^\mathbf{P})/\partial V(I^\mathbf{P})$ to the $d$-cell attached to $\partial T(I)$ yields a dimension reversing bijection between all cells of $V(I^\mathbf{P}) / \partial V(I^\mathbf{P})$ and those of  $T(I) \sqcup_{\partial T(I)} \kappa^{(d)}$.

The next step is to confirm that the face relations between cells in $V(I^\mathbf{P})/\partial V(I^\mathbf{P})$ are mapped to coface relations in $T(I) \sqcup_{\partial T(I)} \kappa^{(d)}$. 
By the construction above, all interior face relations for $V(I^\mathbf{P})$ map to coface relations for $T(I)$. 
Given that only cells in the boundary belong to $[\partial V(I^\mathbf{P})]$, this correspondence is inherited by the quotient. 
    
Hence, the last detail we need to check is that the vertex $[\partial V(I^\mathbf{P})]$ in $V(I^\mathbf{P})/\partial V(I^\mathbf{P})$ has dual face relations to the $d$-cell $\kappa^{(d)}$ attached to the boundary of $T(I)$ in $T(I) \sqcup_{\partial T(I)} \kappa^{(d)}$. 
This is equivalent to the statement that $$ [\partial V(I^\mathbf{P})] \preceq \sigma \text{ in } V(I^\mathbf{P})/\partial V(I^\mathbf{P}) \Leftrightarrow \sigma^* \preceq \kappa^{(d)} \text{ in } T \sqcup_{\partial T} \kappa^{(d)}.$$ 
Now, $\sigma^*$ is a face of $\kappa^{(d)}$ if and only if $\sigma^* \in \partial T(I)$, which means at least one of the degenerate elementary intervals of $\sigma^*$ has $l_i+\tfrac{1}{2} = \tfrac{1}{2}$ or $( n_i + \tfrac{1}{2})$.  This makes $l_i = 0$ or $n_i$, so the corresponding elementary interval in the dual cell $\sigma$ is $e_i = [l_i, l_i+1] = [0,1]$ or $[n_i, n_i+1]$. This forces $\sigma \cap \partial V(I^\mathbf{P}) \neq \emptyset$, so that $[\partial V(I^\mathbf{P})] \preceq \sigma$. The converse implication follows in the same manner, and we are done.
\end{proof}

\begin{lemma}
    Given a rectangular digital image domain, $I$, the quotient, $T(I^{\mathbf{P}})/\partial T(I^{\mathbf{P}})$, of the padded T-construction modulo its boundary is the combinatorial dual of $V(I^\mathbf{P}) \sqcup_{\partial V(I^\mathbf{P})} \kappa^{(d)}$.
\end{lemma}

\begin{proof}
    This follows from the same arguments as the previous lemma with the roles of T and V reversed. Note that we pad the V-construction before attaching the cell $\kappa^{(d)}$ to account for the fact that $T(I)$ naturally has more cells than $V(I)$.
\end{proof}

We have described how the two cubical complex models can be augmented to form dual cell complexes of the $d$-sphere. 
We now show how to obtain dual filtered cell complexes by comparing the 
image function  
on one construction with its negative on the other. 
The details are made precise in the lemmata below.

First note that the function $V(-\mathcal{I}^\mathbf{P})$ is constant on $\partial V(I^\mathbf{P})$ so it induces a function on the quotient space,  $\widetilde{V}(-\mathcal{I}^\mathbf{P}): V(I^\mathbf{P})/\partial V(I^\mathbf{P}) \to \R \text{ with } \widetilde{V}(-\mathcal{I}^\mathbf{P})([\partial V(I^\mathbf{P})]) = -N$ and agreeing with $V(-\mathcal{I}^\mathbf{P})$ on all other cells. 
Similarly, the function $T(\mathcal{I})$ extends to a function $\widehat{T}(\mathcal{I})$ on $T(I) \sqcup_{\partial T(I)} \kappa^{(d)}$ with $\widehat{T}(\mathcal{I})(\kappa^{(d)}) = N$. 

\begin{lemma} \label{dual_filtration_lemma_1}
    For each $\sigma \in T(I) \sqcup_{\partial T(I)} \kappa^{(d)}$ and dual cell $\sigma^* \in V(I^\mathbf{P})/\partial V(I^\mathbf{P})$ we have
    $$ -\widehat{T}(\mathcal{I})(\sigma) = \widetilde{V}(-\mathcal{I}^\mathbf{P})(\sigma^*).$$
\end{lemma}

\begin{proof}
    Firstly, suppose $\dim \sigma = d$ and $\sigma \neq \kappa^{(d)}$, the $d$-cell attached to the boundary. Suppose $p \in I$ is the corresponding element of the image domain, so that $\mathcal{I}(p) = \widehat{T}(\mathcal{I})(\sigma)$. The dual cell $\sigma^* \in V(I^\mathbf{P})/\partial V(I^\mathbf{P})$ corresponds to the same voxel but is given the negative value
    $$ \widetilde{V}(-\mathcal{I}^\mathbf{P})(\sigma^*) = -\mathcal{I}(p) = -\widehat{T}(\mathcal{I})(\sigma).$$ 
    For the remaining $d$-cell $\kappa^{(d)}$, with dual $[\partial V(I^\mathbf{P})]^*$, the function values satisfy  $$-\widehat{T}(\mathcal{I})(\kappa^{(d)}) = -N = \widetilde{V}(-\mathcal{I}^\mathbf{P})([\partial V(I^\mathbf{P})]).$$
    Lastly, suppose $\sigma \in T(I) \sqcup_{\partial T(I)} \kappa^{(d)}$ and $\dim \sigma < d$. By construction, it follows that
    $$-\widehat{T}(\mathcal{I})(\sigma) = -\min_{\tau^{(d)} \succeq \sigma} \widehat{T}(\mathcal{I})(\tau^{(d)}) = \max_{\tau^{(d)} \succeq \sigma}  - \widehat{T}(\mathcal{I})(\tau^{(d)}) = \max_{\upsilon^{(0)} \preceq \sigma^*}  \widetilde{V}(-\mathcal{I}^\mathbf{P})(\upsilon^{(0)}) = \widetilde{V}(-\mathcal{I}^\mathbf{P})(\sigma^*)$$
    as required.
\end{proof}
Now define the functions $\widetilde{T}(-\mathcal{I}^\mathbf{P})$ on $T(I^\mathbf{P})/\partial T(I^\mathbf{P})$ and  $\widehat{V}(\mathcal{I}^\mathbf{P})$ on $V(I^\mathbf{P}) \sqcup_{\partial V(I^\mathbf{P})} \kappa^{(d)}$ similarly to those above. 
\begin{lemma} \label{dual_filtration_lemma_2}
    For each $\sigma \in V(I^\mathbf{P}) \sqcup_{\partial V(I^\mathbf{P})} \kappa^{(d)}$ and dual cell $\sigma^* \in T(I^\mathbf{P})/\partial T(I^\mathbf{P})$ we have
    $$ -\widehat{V}(\mathcal{I}^\mathbf{P})(\sigma) = \widetilde{T}(-\mathcal{I}^\mathbf{P})(\sigma^*).$$
\end{lemma}

\begin{proof}
    Similar to Lemma \ref{dual_filtration_lemma_1} with the roles of V and T interchanged.
\end{proof}
\begin{corollary}
\label{cor:images_dual_filtrations}
For a grayscale digital image $\mathcal{I} : I \to \R$
\begin{enumerate}
    \item The filtered complexes $(T(I) \sqcup_{\partial T(I)} \kappa^{(d)},\widehat{T}(\mathcal{I}))$ and $(V(I^\mathbf{P})/\partial V(I^\mathbf{P}),\widetilde{V}(-\mathcal{I}^\mathbf{P}))$ are dual.
    \item The filtered complexes $(V(I^\mathbf{P}) \sqcup_{\partial V(I^\mathbf{P})} \kappa^{(d)},\widehat{V}(\mathcal{I}^\mathbf{P}))$ and $(T(I^\mathbf{P})/\partial T(I^\mathbf{P}),\widetilde{T}(-\mathcal{I}^\mathbf{P}))$ are dual.
\end{enumerate}
    
\end{corollary}

\begin{proof}
    This follows directly from applying Lemma \ref{dual_filtration_lemma_1} for part (1) and Lemma \ref{dual_filtration_lemma_2} for part (2), then Proposition 
    \ref{negative_function_lemma}.
\end{proof}

\section{Persistence Diagrams of the Modified Filtrations}
\label{sec:technical}

In the previous section, we showed that the T- and V-constructions built from an image can be modified via padding, cell attachment, and taking quotients to become dual cell complexes of the $d$-sphere. Our next step is to examine the effect of such operations on the persistence module and diagram. 

As in Section \ref{sec:grayscale_digital_image_constructions}, suppose $\mathcal{I} : I \to \R$ is a grayscale digital image and  $N > \max \mathcal{I}$. The specific operations we study are
\begin{enumerate} 
    \item Padding an image $\mathcal{I} : I \to \R$ with a outer shell of $N$-valued pixels, then forming the V- and T-constructions. 
    \item Attaching a $d$-cell to the boundary of $V(I^{\mathbf{P}})$ or to $T(I)$ with value $N$.
    \item Taking the quotient modulo the boundary in the negative padded filtration, \emph{i.e.}\ changing from ${V}(-\mathcal{I}^\mathbf{P})$ to $\widetilde{V}(-\mathcal{I}^\mathbf{P})$ and from ${T}(-\mathcal{I}^\mathbf{P})$ to $\widetilde{T}(-\mathcal{I}^\mathbf{P})$. 
\end{enumerate}
Of these, the first two have relatively transparent effects on the persistent homology of the filtered spaces.  Padding the image as in (1) does not change the persistence diagrams; attaching a $d$-cell as in (2) simply creates an essential $d$-cycle with birth at $N$.  We summarise these formally as follows.

\begin{proposition} \label{dcell_padding_proposition}
    For a grayscale digital image $\mathcal{I} : I \to \R$ 
    \begin{enumerate}
    \item $\mathsf{Dgm}(V(\mathcal{I}^\mathbf{P})) =\mathsf{Dgm}(V(\mathcal{I}))$ and $\mathsf{Dgm}(T(\mathcal{I}^\mathbf{P})) =\mathsf{Dgm}(T(\mathcal{I}))$
    \item $$\mathsf{Dgm}(\widehat{V}(\mathcal{I})) = \mathsf{Dgm}(V(\mathcal{I})) \cup \Set{ [N, \infty)_d}$$ and $$\mathsf{Dgm}(\widehat{T}(\mathcal{I})) = \mathsf{Dgm}(T(\mathcal{I})) \cup \Set{ [N, \infty)_d}$$
    \end{enumerate}
\end{proposition}

The remaining operation to investigate is the third, namely the effect of taking the quotient modulo the boundary. For this we need some machinery we will now introduce.

\subsection{Long Exact Sequence of a Filtered Pair}

To examine the effect of taking quotients on persistence diagrams, we use the description of persistence modules as functors together with a long exact sequence (LES) for these. Given a pair of cell complexes $(X,A)$ with $A \subseteq X$, we obtain a short exact sequence (SES) of cellular chain complexes inducing the LES
$$\ldots \to H_k(A) \xrightarrow{i}  H_k(X) \xrightarrow{p} H_k(X,A) \xrightarrow{\delta} H_{k-1}(A) \to \ldots $$
which is a standard tool for analysing the homology of the pair, where $H_k(X,A)$ denotes the relative homology of the pair.

Similarly, suppose we have a filtered cell complex $(X,f)$ and a sub-complex $A \subseteq X$. Then the restriction $f|_A : A \to \R$ induces a filtered sub-complex $(A, f|_A)$ and, at each index $r \in \mathbb R$, we obtain a pair $(X_r,A_r)$, where $X_r = f^{-1}(-\infty, r]$ and $A_r = f|_A^{-1}(-\infty,r]$. Theorem 4.5 in \cite{varl2018homological} states that the $\R$-indexed collection of SESs of cellular chain complexes
$$0 \to C_*(A_r) \to C_*(X_r) \to C_*(X_r,A_r) \to 0$$
yields a LES of persistence modules:

\begin{theorem}[Long Exact Sequence for Relative Persistence Modules] \label{LES4PM}
   Given a monotonic function $f : X \to \R$ and a sub-complex $A \subseteq X$, there is a long exact sequence of persistence modules 
   $$\ldots \to H_k(f|_A) \to H_k(f) \to H_k(f, f|_A) \to H_{k-1}(f |_A) \to H_{k-1}(f) \to \ldots
   $$
   where $H_k(f, f|_A)$ denotes the persistence module given by $r \mapsto H_k(X_r,A_r)$.
\end{theorem}
Here the persistence modules are functors from the poset category $(\R, \leq)$ to  $\mathsf{Vec}_{\Z/2\Z}$ and the morphisms are natural transformations obtained from the standard connecting homomorphisms and the linear transformations induced by inclusions and projections at each filtration index. The kernels and cokernels of the morphisms are themselves persistence modules defined by taking the kernel or cokernel at each filtration index, with maps between corresponding vector spaces at different filtration indices  induced by inclusions. This result is implicit in the recent work in \cite{Bubenik2021, miller2020homological}, where it follows as corollary of the fact that persistence modules form an abelian category whereby the snake lemma holds.

\subsection{Persistence of the Image with Boundary Identified} \label{applications_to_images}
The remaining operation to investigate is the effect of taking the quotient of a padded image modulo the boundary filtered with the negative of the image function. We state the result in terms of a  space $X$ homeomorphic to the $d$-dimensional closed disc $D^d$, filtered by a function $f : X \to \R$ taking a constant minimal value, $\min f = -N$, on the boundary of $X$ so that Lemma~\ref{quotient_lemma} applies to both T- and V-constructions. 
Using the long exact sequence of a pair, we show that the $(d-1)$-cycle with interval $[-N,\max f)_{d-1}$ representing the boundary is removed while a $d$-cycle with interval $[\max f, \infty)_d$ is added.

\begin{lemma} \label{quotient_lemma}
    Take a monotonic function $f : X\cong D^d \to \R$ with $$\sigma \in \D X \Rightarrow f(\sigma) = -N = \min f $$ 
    and induced quotient map $\widetilde{f} : X/\partial X \to \R$. Then
    $$\mathsf{Dgm}(\widetilde{f}) = \big(\mathsf{Dgm}(f)\setminus \Set{ [-N, \max f)_{d-1} }\big) \cup \Set{ [\max f, \infty)_d }.$$
\end{lemma}

\begin{proof}
    For pairs of cell complexes the relative homology groups are naturally isomorphic to the reduced homology groups $\tilde{H}_k(\tilde{f})$ of the quotient \cite[p.124]{hatcher2002}. Naturality implies that the result extends to persistence modules and that the reduced persistence modules differ only by the essential interval  $\mathbb{I}_{[-N, \infty)}$ in degree $0$. To compute the reduced persistence modules $\tilde{H}_k(\tilde{f})$ of the quotient, we therefore consider the LES of the filtered pair $(f, f|_{\D X})$
    $$
    \ldots \to H_{k}(f |_{\D X}) \xrightarrow{\alpha_{k}} H_{k}(f) \to H_{k}(f, f |_{\D X}) \to H_{k-1}(f|_{\D X}) \xrightarrow{\alpha_{k-1}} H_{k-1}(f) \to \ldots
    $$
    where $\alpha_k$ is the map induced by the inclusion $\D X \subseteq X$. Taking the cokernel of $\alpha_{k}$ and the kernel of $\alpha_{k-1}$ the LES yields the SES
    $$
    0 \to \mathsf{Coker}(\alpha_{k}) \to H_{k}(f, f|_{\D X}) \to \mathsf{Ker}(\alpha_{k-1}) \to 0.
    $$
    
    First assume $d > 1$ and note that, in this case,
    $$H_{k}(f|_{\D X}) \cong \begin{cases} \mathbb{I}_{[-N,\infty)} &  \textrm{for} \,\, k = d-1, 0 \\ 0 &  \textrm{otherwise}\end{cases}$$
    
    Thus $\mathsf{Im}(\alpha_k) \cong \mathsf{Ker}(\alpha_k) = 0$ for $k\neq d-1,0$. For $\alpha_{d-1}$, the image of the essential $(d-1)$-cycle of the boundary dies once all cells in $(X,f)$ have been filtered at function value $\max f$. Hence $$\mathsf{Im}(\alpha_{d-1}) \cong \mathbb{I}_{[-N, \max f)}  \quad \textrm{and} \quad \mathsf{Ker}(\alpha_{d-1}) \cong \mathbb{I}_{[\max f, \infty)}$$
    As $-N = \min f$ we conclude $\alpha_0(\mathbb{I}_{[-N, \infty)}) = \mathbb{I}_{[-N, \infty)}$, so that
    $$\mathsf{Im}(\alpha_{0}) \cong \mathbb{I}_{[-N, \infty)}  \quad \textrm{and} \quad \mathsf{Ker}(\alpha_{0}) = 0$$
    
    Since $X$ is homeomorphic to a $d$-dimensional disc, $H_d(f) = 0$. Hence $\mathsf{Coker}(\alpha_{d}) = 0$ and, for $k = d$, the SES implies
    $$
    \tilde H_d(\tilde{f}) \cong H_d(f, f|_{\D X}) \cong \mathsf{Ker}(\alpha_{d-1}) \cong \mathbb{I}_{[\max f, \infty)}.
    $$
    For $0 \leq k < d$ the persistence module on the right of the SES is trivial. Thus
    $$
    \tilde H_{k}(\tilde f) \cong H_{k}(f, f|_{\D X}) \cong \mathsf{Coker}(\alpha_{k}) \cong \begin{cases} H_{d-1}(f) / \mathbb{I}_{[-N, \max f)} & \textrm{for} \,\, k = d-1 \\ H_{k}(f)& \textrm{for} \,\, 0< k < d-1 \\ H_{0}(f) / \mathbb{I}_{[-N, \infty)} & \textrm{for} \,\, k = 0 \end{cases}
    $$
    and the result follows for $d > 1$.
    
    For $d = 1$ the only non-trivial persistence module of the boundary is $H_{0}(f|_{\D X}) \cong \mathbb{I}_{[-N,\infty)} \oplus \mathbb{I}_{[-N,\infty)}$
    and we obtain 
    $$\mathsf{Im}(\alpha_0) \cong \mathbb{I}_{[-N,\infty)} \oplus \mathbb{I}_{[-N,\max f)} \quad \textrm{and} \quad \mathsf{Ker}(\alpha_0) \cong \mathbb{I}_{[\max f,\infty)}$$
    For $k = 1$ we proceed as above and, for $k = 0$, the SES yields 
    $$\tilde H_0(\tilde f) \cong H_0(f, f|_{\D X}) \cong \mathsf{Coker}(\alpha_{0}) \cong H_{0}(f) / \big(\mathbb{I}_{[-N, \infty)} \oplus \mathbb{I}_{[-N, \max f)}\big)$$
    As above we conclude $H_0(\tilde f) \cong H_0(f) / \mathbb{I}_{[-N, \max f)}$.
\end{proof}

\begin{corollary} \label{img_dual_cor} For a $d$-dimensional image $\mathcal{I} : I \to \R$
    $$\mathsf{Dgm}(\widetilde{V}(-\mathcal{I}^{\mathbf{P}})) = \mathsf{Dgm}(V(-\mathcal{I}^{\mathbf{P}})) \setminus \Set{ [-N, - \min \mathcal{I})_{d-1} } \cup \Set{ [- \min \mathcal{I}, \infty)_d }$$
and
    $$\mathsf{Dgm}(\widetilde{T}(-\mathcal{I}^{\mathbf{P}})) = \mathsf{Dgm}(T(-\mathcal{I}^{\mathbf{P}})) \setminus \Set{ [-N, - \min \mathcal{I})_{d-1} } \cup \Set{ [- \min \mathcal{I}, \infty)_d }$$
\end{corollary}

\begin{proof}
    This follows from Lemma~\ref{quotient_lemma} applied to $f=V(-\mathcal{I}^\mathbf{P})$ and $f=T(-\mathcal{I}^\mathbf{P})$ respectively, using $\max V(-\mathcal{I}^\mathbf{P}) = - \min \mathcal{I}$ and $\max T(-\mathcal{I}^{\mathbf{P}})=- \min \mathcal{I}$.
\end{proof}

\section{Duality Results for Images}
\label{image_results}
In this section, we explicitly describe the relationship between the diagrams of both the T- and V-constructions.
Software to compute persistent homology of an image $\mathcal{I} : I \to \R$ typically builds one of the two constructions implicitly so the results in this section provide a solution to the problem of how to use software based on the V-construction to compute a persistence diagram with respect to the T-construction, and vice versa.


For the algorithms that we define in this section we assume the following sub-routines given a grayscale digital image $\mathcal{I}$.

\begin{enumerate}
    \item $\textsc{Pad}(\mathcal{I}, N)$: returns the image padded with an outer shell of $N$-valued voxels.
    \item $\textsc{Neg}(\mathcal{I})$: multiplies every gray value by $-1$. 
    \item $\max(\mathcal{I})$, $\min(\mathcal{I})$: returns the maximum and minimum voxel values of $\mathcal{I}$ respectively.
    \item \textsc{Vcon}$(\mathcal{I})$, \textsc{Tcon}$(\mathcal{I})$: returns the persistence diagrams $\mathsf{Dgm}(V(\mathcal{I}))$ and $\mathsf{Dgm}(T(\mathcal{I}))$ of the V- and T-construction of the image respectively.
\end{enumerate}

\subsection{From the V-construction to the T-construction}
\label{sec: from_V_to_T}

Suppose we have software that computes the persistent homology of a $d$-dimensional grayscale digital image $\mathcal{I}$ using the V-construction. 
The following theorem states that the persistence diagram of the T-construction for $\mathcal{I}$ can be calculated directly from the pairs in that of the V-construction of the negative padded image.

\begin{theorem}[T from V]\label{from_V_to_T}
    For a grayscale digital image $\mathcal{I} : I \to \R$ the diagrams of the V- and T-constructions satisfy 
    $$\mathsf{Dgm}_\mathbf{F}(T(\mathcal{I})) = \Set{ [-q,-p)_{d-k-1} \mid [p,q)_k \in \mathsf{Dgm}_\mathbf{F}(V(-\mathcal{I}^\mathbf{P})) } \setminus \Set{ [ \min \mathcal{I}, N)_0 }$$
    and
    $$\mathsf{Dgm}_\infty(T(\mathcal{I}))=\Set{ [\min \mathcal{I}, \infty)_0}.$$
\end{theorem}

\begin{proof}
    That $\mathsf{Dgm}_{\infty} (T(\mathcal{I})) = \Set{ [\min \mathcal{I}, \infty)_0}$ follows from the fact that $T(I) \cong D^{(d)}$ and the first cell in the filtration occurs at time $\min T(\mathcal{I}) = \min \mathcal{I}$. For the finite case:
    \begin{align*}
    \mathsf{Dgm}_\mathbf{F}(T(\mathcal{I})) & = \mathsf{Dgm}_\mathbf{F}(\widehat{T}(\mathcal{I})) \\
    & =\Set{ [-q,-p)_{d-k-1} \mid [p,q)_k \in \mathsf{Dgm}_\mathbf{F}(\widetilde{V}(-\mathcal{I}^\mathbf{P}))}\\
    & = \Set{ [-q,-p)_{d-k-1} \mid [p,q)_k \in \mathsf{Dgm}_\mathbf{F}(V(-\mathcal{I}^\mathbf{P})) \setminus \Set{[-N, -\min  \mathcal{I})_{d-1}}} \\
    & = \Set{ [-q,-p)_{d-k-1} \mid [p,q)_k \in \mathsf{Dgm}_\mathbf{F}(V(-\mathcal{I}^\mathbf{P})) } \setminus \Set{[\min \mathcal{I}, N)_0}
    \end{align*}
    where the equalities follow from Proposition \ref{dcell_padding_proposition}, Theorem \ref{main_theorem} and Corollary \ref{img_dual_cor} respectively.
\end{proof}

\noindent The structure of the algorithm follows immediately from the theorem and is summarised below:\\

\begin{algorithm}
\setstretch{1.3}
\caption{Computing the T-construction persistence diagram with V-construction software.}\label{alg:VtoT}
\begin{algorithmic}[1]
\Require{ An image $\mathcal{I}$ and the sub-routine $\textsc{Vcon}$.}
\State $\mathsf{Dgm}(T(\mathcal{I})) \gets \Set{ [ \min (\mathcal{I}), \infty)_0 }$ 
\State $N \gets \max(\mathcal{I}) + C$ \Comment{choose $C$ to ensure $N \gg \max(\mathcal{I})$}
\State $-\mathcal{I}^\mathbf{P} \gets \textsc{Neg}(\textsc{Pad}(\mathcal{I},N))$
\State $\mathsf{Dgm}(V(-\mathcal{I}^\mathsf{P})) \gets \textsc{Vcon}( -\mathcal{I}^\mathbf{P})$ \Comment{Apply V-construction software.}
\For{ $[p,q)_k$ in $\mathsf{Dgm}(V(-\mathcal{I}^\mathbf{P}))$ with $p\neq -N$} 
\State $\mathsf{Dgm}(T(\mathcal{I})) \gets \mathsf{Dgm}(T(\mathcal{I})) \cup \Set{ [-q, -p)_{d-k-1}}$
\EndFor
\State \textbf{return} $\mathsf{Dgm}(T(\mathcal{I}))$ \Comment{Output T-construction persistence diagram.}
\end{algorithmic}
\end{algorithm}

\subsection{From the T-construction to the V-construction}
\label{sec: T_to_V}

In the other direction, suppose we have software that computes the persistent homology of a $d$-dimensional grayscale digital image $\mathcal{I} : I \to \R$ using the T-construction. The following theorem states that a persistence diagram for the V-construction of $\mathcal{I}$ can be calculated directly from the pairs computed for the negative padded image using the T-construction.

\begin{theorem}[V from T] \label{from_T_to_V}
    For a grayscale digital image $\mathcal{I} : I \to \R$ the diagrams of the V- and T-constructions satisfy $$\mathsf{Dgm}_\mathbf{F}(V(\mathcal{I})) = \Set{ [-q,-p)_{d-k-1} \mid [p,q)_k \in \mathsf{Dgm}_\mathbf{F}(T(-\mathcal{I}^\mathbf{P}))} \setminus \Set{ [ \min \mathcal{I}, N)_0 }$$ 
    and
    $$\mathsf{Dgm}_\infty(V(\mathcal{I}))= \Set{ [\min \mathcal{I}, \infty)_0}.$$ 
\end{theorem}

\begin{proof}
    That $\mathsf{Dgm}(V(\mathcal{I})) = \Set{ [\min \mathcal{I}, \infty)_0}$ follows from the fact that $V(I) \cong D^{(d)}$ and the first cell in the filtration occurs at time $\min \mathcal{I}$. For the finite case, we have that 
    \begin{align*}
    \mathsf{Dgm}_\mathbf{F}(V(\mathcal{I})) & = 
    \mathsf{Dgm}_\mathbf{F}(V(\mathcal{I}^\mathbf{P})) \\
    & = \mathsf{Dgm}_\mathbf{F}(\widehat{V}(\mathcal{I}^\mathbf{P})) \\
    & =\Set{ [-q,-p)_{d-k-1} \mid [p,q)_k \in \mathsf{Dgm}_\mathbf{F}(\widetilde{T}(-\mathcal{I}^\mathbf{P}))}\\
    & = \Set{ [-q,-p)_{d-k-1} \mid [p,q)_k \in \mathsf{Dgm}_\mathbf{F}(T(-\mathcal{I}^\mathbf{P})) \setminus \Set{[-N, -\min  \mathcal{I})_{d-1}}} \\
    & = \Set{ [-q,-p)_{d-k-1} \mid [p,q)_k \in \mathsf{Dgm}_\mathbf{F}(T(-\mathcal{I}^\mathbf{P})) } \setminus \Set{[\min \mathcal{I}, N)_0}
    \end{align*}
     where the equalities follow from Proposition \ref{dcell_padding_proposition}, Theorem \ref{main_theorem} and Corollary \ref{img_dual_cor} respectively.
\end{proof}

\noindent The structure of the algorithm follows immediately from the theorem and is summarised below:\\

\begin{algorithm}
\setstretch{1.3}
\caption{Computing the V-construction persistence diagram with T-construction software.}\label{alg:TtoV}
\begin{algorithmic}[1]
\Require{ An image $\mathcal{I}$ and the sub-routine $\textsc{Tcon}$.}
\State $\mathsf{Dgm}(V(\mathcal{I})) \gets \Set{ [ \min (\mathcal{I}), \infty)_0 }$
\State $N \gets \max(\mathcal{I}) + C$ \Comment{choose $C$ to ensure $N \gg \max(\mathcal{I})$}
\State $-\mathcal{I}^\mathbf{P} \gets \textsc{Neg}(\textsc{Pad}(\mathcal{I},N))$
\State $\mathsf{Dgm}(T(-\mathcal{I}^\mathsf{P})) \gets \textsc{Tcon}( -\mathcal{I}^\mathbf{P})$ \Comment{Apply T-construction software.}
\For{ $[p,q)_k$ in $\mathsf{Dgm}(V(-\mathcal{I}^\mathbf{P}))$ with $p\neq -N$} 
\State $\mathsf{Dgm}(V(\mathcal{I})) \gets \mathsf{Dgm}(V(\mathcal{I})) \cup \Set{ [-q, -p)_{d-k-1}}$
\EndFor
\State \textbf{return} $\mathsf{Dgm}(V(\mathcal{I}))$ \Comment{Output V-construction persistence diagram.}
\end{algorithmic}
\end{algorithm}

\begin{example} Suppose we are working with the two dimensional digital grayscale image given in Figure \ref{fig:V_and_T_constructions} and have only the software to compute the T-construction. We depict the filtration of $T(-\mathcal{I}^\mathbf{P})$ in Figure $\ref{dual_barcode}$, and the corresponding intervals in the persistence module. Similarly, we show the filtered V-construction $V(\mathcal{I})$ in Figure \ref{fig:V-construction_filtration}. The reader may confirm that the correspondence between the intervals is accurately described by Theorem \ref{from_T_to_V}.

\begin{figure}[H]
\includegraphics[width=\linewidth]{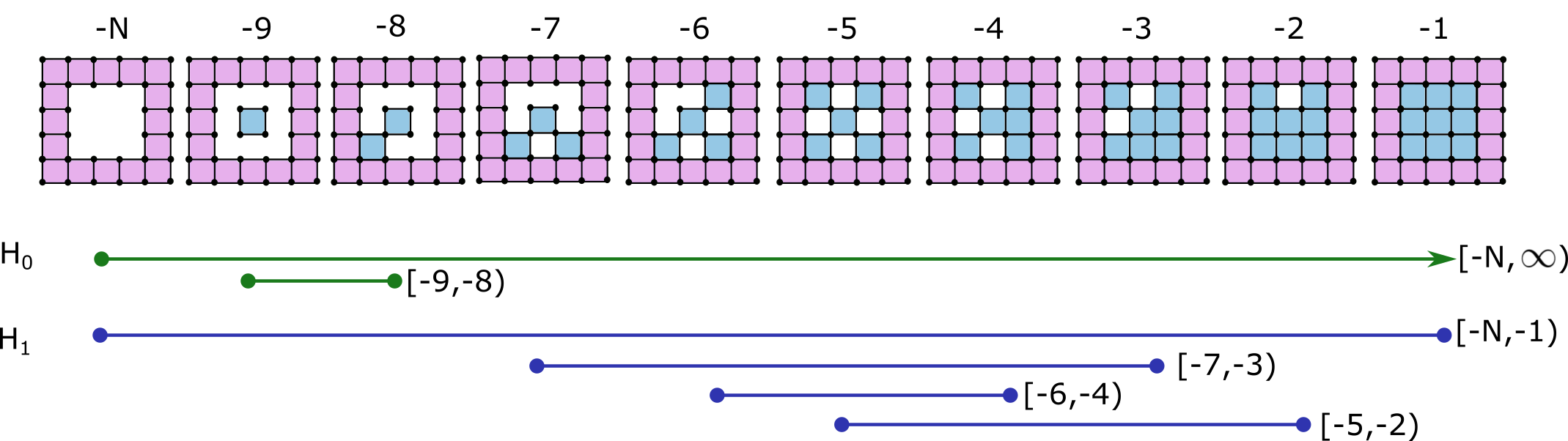}
    \caption{The filtration $T(-\mathcal{I}^\mathbf{P})$ and intervals of the persistence diagram $\mathsf{Dgm}(T(-\mathcal{I}^\mathbf{P}))$ for the image $\mathcal{I} : I \to \R$ of Figure \ref{fig:V_and_T_constructions}. 
    }
    \label{dual_barcode}
\end{figure}

\begin{figure}[H]
    \centering
    \includegraphics[width=\linewidth]{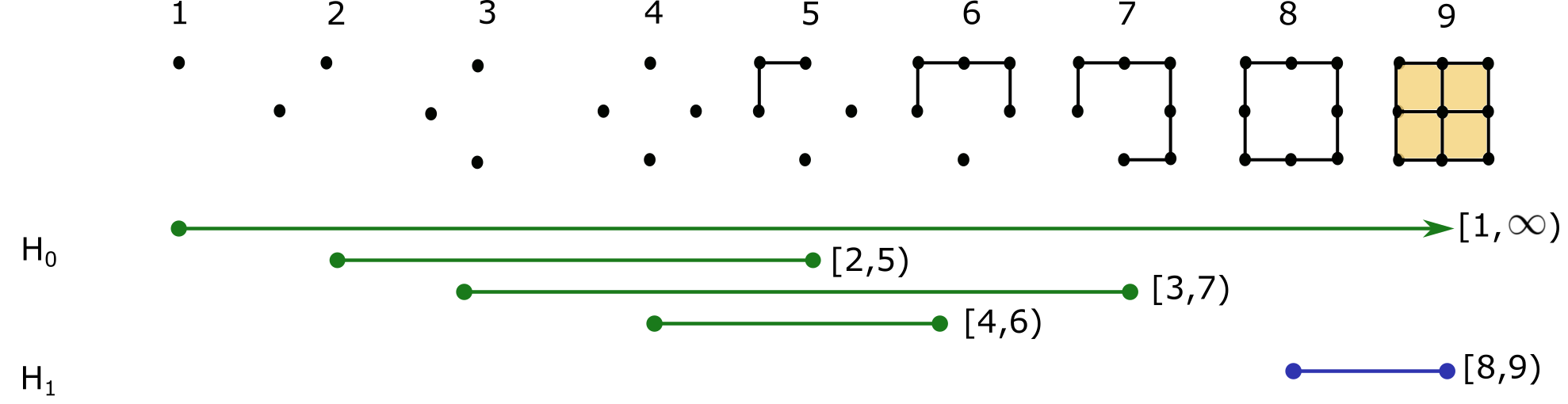}
    \caption{The filtration $V(\mathcal{I})$ and intervals of the persistence diagram $\mathsf{Dgm}(V(\mathcal{I}))$ for the image $\mathcal{I} : I \to \R$ of Figure \ref{fig:V_and_T_constructions}.}
    \label{fig:V-construction_filtration}
\end{figure}
\end{example}

\section{Discussion}

Our results clarify the relationship between the two cubical complex constructions commonly used in digital image analysis software and provide a simple method to use software that implements one construction to compute a persistence diagram for the other. 
This permits a user's choice of adjacency type for their images to depend on that appropriate to the application rather than on the type of construction used in available efficient persistence software. 
In addition to facilitating this application, the results of Sections~\ref{sec:main_duality_result} and~\ref{sec:technical} may be of independent interest for the following reasons. 

Theorem \ref{main_theorem} is a new interpretation of a duality relationship that manifests in many contexts such as the correspondence between persistent homology and persistent relative cohomology~\cite{Vin}, symmetries in extended persistence diagrams~\cite{extending_pers}, and a discrete Helmoltz-Hodge decomposition~\cite{oelsboeck}.
In \cite{socg}, we show that the filtered discrete Morse chain complexes also exhibit this duality.  Further investigations in this area may reveal other interpretations of this relationship.

The results of Section~\ref{sec:technical} are formulated specifically for the case of an image with its domain homeomorphic to a closed ball, but could be extended to spaces with more interesting topology.  We anticipate that the long exact sequence of a pair can be used to derive a relationship between filtered cell complexes that satisfy conditions for duality if their boundaries can be capped or quotiented as in Section~\ref{sec:grayscale_digital_image_constructions} to obtain a manifold.   
 
Further applications of the relationship between the T- and V-constructions can be found.  In particular, a discrete gradient vector field built on $V(I)$ is easily transformed into a dual one on $T(I)$~\cite{socg}.  If the gradient vector field is built to be consistent with the grayscale image $\mathcal{I}$ according to the algorithm of~\cite{Robins_DMT_images}, then the dual gradient vector field on $T(I)$ will be consistent with $-\mathcal{I}$.  
This means the skeletonisation and partitioning algorithms of~\cite{delgado2015skeletonization} can now be adapted to work on images where the T-construction is preferred. 

There are also interesting questions about algorithm performance to explore.  
The results of Section \ref{image_results} suggest that persistence diagram computation from grayscale images should have the same average run time independent of the choice of T- and V-construction.  If the T-construction executes faster on a particular image, then the V-construction should execute faster on the negative of the image.  To answer this question fully requires a careful analysis of the effects of taking the anti-transpose of the boundary matrix on the run time of the matrix reduction algorithm and the extra cells added when padding the image.

\section{Acknowledgments}
This project started during the Women in Computational Topology workshop held in Canberra in July of 2019. All authors are very grateful for its organisation and the financial support for the workshop from the Mathematical Sciences Institute at ANU, the NSF, AMSI and AWM. AG is supported by the Swiss National Science Foundation, grant No.\ $CRSII5\_177237$. TH is supported by the ERC Horizon 2020 project ‘Alpha Shape Theory Extended’, No.\ 788183. KM is supported by the European Union’s Horizon 2020 research and innovation programme under the Marie Skłodowska-Curie grant agreement No.\ 859860. VR was supported by ARC Future Fellowship FT140100604 during the early stages of this project.

\bibliography{biblio.bib}

\appendix

\end{document}